\documentclass [a4paper, 12pt]{amsart}

\usepackage[margin=2.3cm]{geometry} 
\usepackage{amsrefs}
\usepackage{times}
\usepackage{wasysym,stmaryrd,enumerate}
\usepackage{amssymb,amsfonts,amsrefs,amsthm,mathrsfs,amsmath,amscd,version,graphicx,bbm, xspace}
\usepackage{tikz-cd}
\usepackage{tikz}
\usepackage{amsmath}
\usepackage{amsfonts}
\usepackage{amssymb}
\usepackage{amsthm}
\usepackage{graphicx}
\usepackage{pb-diagram}
\usepackage{epstopdf}
\usepackage{fancyhdr}
\usepackage[all]{xy}

\newtheorem{theo}{Theorem}[section]
\newtheorem{theorem}[theo]{Theorem}
\newtheorem{corollary}[theo]{Corollary}

\newtheorem{lemma}{Lemma}
\newtheorem{propo}{Proposition}

\newtheorem{theor}{Theorem}

\theoremstyle{definition}

\newtheorem{defii}{Definition}

 \newtheorem{remark}{Remark}
	
\newtheorem*{question}{Questions}	
\newtheorem*{questions}{Question}	
 \newtheorem{definition}{Definition}

\DefineSimpleKey{bib}{arx}
\BibSpec{arx}{%
  +{}{\PrintAuthors} {author}
  +{,}{ \textit} {title}
  +{}{ \parenthesize} {date}
  +{,}{ } {note}
  +{.}{ } {transition}
}

\begin{document}

\title{Computability of F\o lner sets}

\author{Matteo Cavaleri}
\address[M. Cavaleri]{Institute of Mathematics of the Romanian Academy, 21 Calea Grivitei Street, 010702 Bucharest, Romania}
\email{matte.cavaleri@gmail.com}
\subjclass[2010]{20F10, 03D40, 43A07}
\thanks{This work was partially supported by a grant of the Romanian National Authority for Scientific Research and
Innovation, CNCS - UEFISCDI, project number PN-II-RU-TE-2014-4-0669, and partially supported by the European Research Council (ERC) grant ANALYTIC no. 259527 of Prof. Goulnara Arzhantseva.}

\begin{abstract}
We define the notion of computability of F\o lner sets for finitely generated amenable groups. We prove, by an explicit description, that the Kharlampovich groups,  finitely presented solvable groups with unsolvable Word Problem, have computable F\o lner sets. We also prove computability of F\o lner sets for  extensions -with subrecursive distortion functions- of amenable groups with solvable Word Problem by finitely generated groups with computable F\o lner sets . Moreover we obtain some known and some new upper bounds for the F\o lner function for these particular extensions.
\end{abstract}

\maketitle

\section*{Introduction}
In this paper we define and study an effective version of amenability for finitely generated groups in terms of computability of F\o lner sets. 
Let $\Gamma$ be a group generated by a finite set $X$. We denote by $\pi_{\Gamma,X}\colon\mathbb F_X\rightarrow\Gamma$ the canonical epimorphism from the free group over $X$ to $\Gamma$. For any $n\in \mathbb N$, an \emph{$n$-F\o lner set} for $\Gamma$ (with respect to $X$) is defined to be a non-empty finite subset $F\subset \Gamma$ such that 
\begin{equation}\label{folner}\frac{|F\setminus x F|}{|F|}\leq n^{-1}, \;\;\;\forall x\in X.\end{equation}
 $\Gamma$ is \emph{amenable} if it admits $n$-F\o lner sets, for all $n\in\mathbb N$. In order to make \emph{effective} this notion,  we simply ask for computability of finite preimages of F\o lner sets in the covering free group.
\begin{defii}\label{prima}

$\Gamma$ has \emph{computable F\o lner sets} with respect to $X$ if there exists an algorithm with:\\
INPUT: $n\in \mathbb N$\\
OUTPUT: $F \subset \mathbb F_X$ finite, such that $\pi_{\Gamma,X}(F)$ is $n$-F\o lner for $\Gamma$.
\end{defii}
\noindent This definition does not depend on the particular choice of the finite set of generators (see Proposition \ref{inva}).

The \emph{F\o lner function} $F_{\Gamma,X}$ of $\Gamma$ (with respect to $X$) was defined by Vershik \cite{VER} by
$$F_{\Gamma,X}(n):=\min\{|F|:\;\; F\subset \Gamma \mbox{  is $n$- F\o lner}\}.$$

 The aim of our investigation, and in particular of Definition \ref{prima}, 
  is to formalize and answer an old question, in our notation:
  \begin{questions}[Vershik]
   ``is it possible, in some sense, to (algorithmically) describe the $n$-Folner sets of $\Gamma$ even if there is no solution for the Word Problem?"
   \end{questions}
   This question arose after the construction by Kharlampovich of finitely presented groups, solvable and therefore amenable, with unsolvable Word Problem \cite{H} (following the notation in \cite{KMS}, let denote them by $G(M)$).
  
Indeed a finitely generated amenable group with solvable Word Problem has computable F\o lner sets: for every $n\in \mathbb N$ we can enumerate all finite subsets of $\mathbb F_X$ and for each subset check, by solvability of the Word Problem, condition \eqref{folner}, until we find the preimage of an $n$-F\o lner set:  the algorithm will eventually stop because $\Gamma$ is amenable (for details \cite{pre,CAV}).

A first negative answer to Vershik's question was given by Erschler \cite{A2}, providing examples of finitely generated groups with F\o lner functions growing faster than any given function (a result recovered in \cite{G,OO}): thus, when the given function is not subrecursive (i.e.\ without any recursive upper bound) there is no hope to algorithmically describe F\o lner sets of the associated groups (for details \cite{pre,CAV}).

On the other hand, in \cite{pre}, we proved that, if $\Gamma$ is recursively presentable and amenable then\\
(i) the F\o lner function of $\Gamma$ is subrecursive;\\
(ii) there exists an algorithm computing the Reiter functions of $\Gamma$;\\
(iii) computability of one-to-one preimages of F\o lner sets is equivalent to solvability of the Word Problem.\\
As a consequence, for recursively presented amenable groups, the notion of computability of F\o lner sets is the only non trivial notion of effective amenability that is not characterized by solvability of the Word Problem.

 Indeed computability of F\o lner sets does not imply solvability of the Word Problem: this is the first consequence of the following  theorem, proved  by a very explicit description of the F\o lner sets of $G(M)$ (see Section \ref{K}).

\begin{theor}\label{KA}
 The Kharlampovich groups $G(M)$ have computable F\o lner sets. 
\end{theor}

Amenability is stable under semidirect products and, more generally, under extensions: in the literature, the most common proofs of these two facts do not use the characterization of amenability by F\o lner sets. The book \cite{Coo} is one of the exceptions and it was a valuable resource for our proofs. Moreover in \cite{J} it was explicitly  shown that a F\o lner net for the semidirect product is given by the product of the F\o lner nets of the factor groups. However this does not yield an effective procedure to produce, for a fixed $n\in \mathbb N$, an $n$-F\o lner set.

After the Preliminaries, each section consists of a Theorem about the general shape of $n$-F\o lner sets of group extensions, a Corollary about computability of these F\o lner sets and a Corollary about the F\o lner functions.
We can interpret the results of this paper as stability properties of the class of groups with computable F\o lner sets and of the class of groups with subrecursive F\o lner functions.

\subsubsection*{Section \ref{K}} We consider the case of a splitting extension by an Abelian group which is finitely generated as a normal subgroup (this is the case of $G(M)$): computability of F\o lner sets and subrecursivity of F\o lner function are preserved.

\subsubsection *{Section \ref{A} } We consider general Abelian extensions: subrecursivity of F\o lner functions is preserved but we prove computability of F\o lner sets just if the quotient group has solvable Word Problem. We don't know if this hypothesis is necessary.
Asymptotically equivalent bounds for the F\o lner function of solvable groups could be also deduced from  \cite{A1, A2}, or using the comparison with the F\o lner function in free solvable groups in \cite{SC}. 

\subsubsection*{Section \ref{S}} We consider the semidirect product between two finitely generated groups: if both have computable F\o lner sets then the product has computable F\o lner sets.

\subsubsection*{Section \ref{G}} We consider an extension $\Gamma$ of a finitely generated group $K$ by a finitely generated group $N=\langle Y \rangle$. If $K$ and $N$ have subrecursive F\o lner functions and the
 \emph{distortion function} 
${\Delta_N^{\Gamma}(n):=
\max\{|\omega|_Y:\, \omega \in N,\, 
|\omega|_X \leq n\}}$  is subrecursive,
 then $\Gamma$ has subrecursive F\o lner
 function; if $N$ has computable F\o lner
 sets, $K$ is amenable with solvable word 
problem and $\Delta_N^{\Gamma}$ is subrecursive
  then $\Gamma$ has computable F\o lner sets.
	Notice that it is possible that $\Delta_N^\Gamma$
	is not subrecursive, see for example \cite{AO}, even for solvable groups, see \cite{OD}.
Again we don't know if these hypotheses are necessary.

\begin{question}${}$

\noindent
1. Have all finitely generated solvable groups computable F\o lner sets?

\noindent
2. Is computability of F\o lner sets stable under quotients?

\noindent 3. Does subrecursivity of the F\o lner function imply computability of F\o lner sets?
\end{question}
A positive answer to the third question would imply both  computability of F\o lner sets for every recursively presented amenable group and a positive answer to the second question, because subrecursivity of the F\o lner function is stable under quotients (see \cite[Lemma 2.2]{A2}); a positive answer to the second question would imply a positive answer to the first one, because free solvable groups have solvable Word Problem and therefore have computable F\o lner sets.

\subsection*{Acknowledgements}
This work started from my PhD program in Sapienza Universit\`a di Roma, under the supervision of Tullio Ceccherini-Silberstein.
Part of this work was developed during the program \textit{Measured group theory} in Vienna in 2016. 
We thank its organizers, as well as the Erwin Schr\"{o}dinger Institute for Mathematics and Physics, for the warm hospitality. We also thank gratefully Federico Berlai for an improvement in Theorem \ref{abex} and Theorem \ref{genex}, Michel Coornaert for precious comments,  Goulnara Arzhantseva and Anatoly Vershik for their interest and encouragement.

\section*{Preliminaries}

Throughout this paper $\mathfrak F\o l_{\Gamma,X}(n)$ is the family of $n$-F\o lner sets of $\Gamma$ with respect to $X$. For an element $g\in \Gamma$ we denote with $|g|$ the length with respect to $X\cup X^{-1}$ (so it is the minimal length of a word in $\mathbb F_X$ representing $g$). For a different set of generators, say $Y$, we explicitly write $|g|_Y$. For a subset $A$ we also denote with $|A|_Y$ the maximal length of the elements of $A$ with respect to $Y$. We denote by $B_n$ the ball of radius $n$ in the free group and $B_n(\Gamma):=\pi_{\Gamma,X}(B_n)$ the ball of radius $n$ of $\Gamma$.

\begin{lemma}\label{catena}
For any $F\in\mathfrak F\o l_{\Gamma, X}(n)$ and for all $g\in \Gamma$ we have: 

$$\frac{|F\setminus g F|}{|F|}\leq |g| n^{-1}.$$

\end{lemma}
\begin{proof}
At first we observe that if $F\in \mathfrak F\o l_{\Gamma, X}(n)$, for every $x\in X$ we have:
$$\frac{|F\setminus x^{-1} F|}{|F|}=\frac{|x^{-1}(xF\setminus F|)}{|F|}=\frac{|xF\setminus  F|}{|F|}=\frac{|F\setminus x F|}{|F|}\leq n^{-1}.$$

If $g=x_1\ldots x_{|g|}$, with $x_1,\ldots, x_{|g|} \in X\cup X^{-1}$,
$$(F\setminus x_1\ldots x_{|g|} F)\subset [(F\setminus x_1 F) \cup (x_1 F\setminus x_1 x_2 F)\cup  \ldots \cup (x_1\ldots x_{|g|-1} F\setminus x_1\ldots x_{|g|} F)],$$ 
and $|x_1\ldots x_{j-1} F\setminus x_1\ldots x_j F|=|F\setminus x_j F|.$
\end{proof}

\begin{propo}\label{inva}
Suppose $\Gamma$ has computable F\o lner sets with respect to $X$. Let  $Y\subset \Gamma$  be another finite generating subset. Then $\Gamma$ has computable F\o lner sets with respect to $Y$ as well.
\end{propo}
\begin{proof}
By expressing every $x\in X$ in terms of words in $Y$, we define a homomorphism ${\phi\colon \mathbb F_{X}\to \mathbb F_{Y}}$ such that the following diagram commutes:

\[
  \begin{tikzcd}
     \mathbb F_X \arrow{r}{\phi} \arrow[swap]{dr}{\pi_{\Gamma,X}} & \mathbb F_Y \arrow{d}{\pi_{\Gamma,Y}} \\ & \Gamma
  \end{tikzcd}
\]
We define the natural number $m:=|\phi(X)|_Y$ as the maximum of the word length in $\mathbb F_Y$ of the image of generators of $\mathbb F_X$. Suppose $W\subset \mathbb F_X$ is a finite subset such that $\pi_{\Gamma,X}(W)\in \mathfrak F\o l_{\Gamma,X}(mn)$; by Lemma \ref{catena} $\pi_{\Gamma,Y}(\phi(W))=\pi_{\Gamma,X}(W)\in \mathfrak F\o l_{\Gamma,Y}(n)$. Combining the algorithm of Definition \ref{prima} for $X$, 
with the algorithm computing the homomorphism $\phi$, we deduce the existence of the desired algorithm for $Y$.
\end{proof}
\begin{remark}
In finitely presented case (by Tietze transformations), or, more generally, whenever we can write the old generators in terms of the new ones, we can explicitely update the algorithm of Definition \ref{prima} for the new generators.
\end{remark}

\begin{lemma}\label{conn}
If $\Gamma$ is amenable then there exists $F\in \mathfrak F\o l_{\Gamma,X}(n)$ such that $|F|\leq F_{\Gamma, X}(|X|n)$ and $F\subset B_{|F|}(\Gamma).$
\end{lemma}
\begin{proof}
We define $$\mathfrak F\o l'_{\Gamma,X}(n):=\{F, \mbox{ non-empty finite subset of  } \Gamma:\;\;  \frac{|\partial_X F|}{|F|}\leq \frac{1}{n} \},$$ where 
$\partial_X F:=\{f\in F: \exists \, x\in X: xf\notin F\}.$ 

It is known and easy to see that if $F'$ is of minimal cardinality in $\mathfrak F\o l'_{\Gamma,X}(n)$ (\emph{optimal F\o lner set}) then it is connected as subgraph of the right Cayley graph of $\Gamma$ with respect to $X$ (see \cite{CAV} for details). In particular for $f\in F'$ we have that $1_{\Gamma}\in F:=F'f^{-1}$ and  $F\subset B_{|F|}(\Gamma)$ and $F\in \mathfrak F\o l'_{\Gamma,X}(n)$.
Finally since:
$$ \mathfrak F\o l_{\Gamma,X}(|X|n) \subset \mathfrak F\o l'_{\Gamma,X}(n)\subset \mathfrak F\o l_{\Gamma,X}(n)$$
we have that $F\in \mathfrak F\o l_{\Gamma,X}(n)$ and $|F|\leq F_{\Gamma, X}(|X|n)$.

\end{proof}
\begin{definition}

Let $y_1, y_2, \ldots , y_{s}$ be pairwise commuting elements of $\Gamma$, not ne\-ces\-sa\-ri\-ly distinct. Set: 
$$C_n(y_1,y_2,\ldots , y_{s}):=\{y_1^{i_1}y_2^{i_2}\ldots y_{s}^{i_{s}}:\;\;\;i_1,i_2,\dots, i_{s}\in\{0,1,\ldots, n-1\}\}.$$
\end{definition}

\begin{lemma}\label{cubo}

\[\frac{|C_n(y_1,y_2,\ldots,  y_{s})\setminus y_j C_n(y_1,y_2,\ldots,  y_{s})|}{|C_n(y_1,y_2,\ldots,  y_{s})|}\leq n^{-1},\;\;\;\forall j\in \{1,2,\ldots,  d\}.\]
\end{lemma}
\begin{proof}
Since all elements $y_1, y_2, \ldots , y_{s}$ commute we prove, without loss of generality,  the statement for $j=1$.\\
At first, we observe that
$C_n(y_1,y_2,\ldots , y_{s})=C_n(y_1) C_n(y_2,y_3,\ldots , y_{s})$
and 
\begin{equation*}
C_n(y_1)\setminus y_1 C_n(y_1)=\begin{cases}
\emptyset \; \mbox{ if $y_1$ has order less than or equal to } n \\
\{1_\Gamma \}\;\mbox{ otherwise. }
\end{cases}
\end{equation*}
Writing $C_n$ instead of $C_n(y_1,y_2,\ldots , y_{s})$ we have that $$C_n \setminus y_1 C_n\subset C_n(y_2,y_3,\ldots,y_{s}),$$ because $C_n \setminus y_1 C_n\subset [C_n(y_1)\setminus y_1 C_n(y_1)] C_n(y_2,y_3,\ldots,y_{s})$.

Now we show that $C_n$ contains $n$ disjoint translations of $C_n\setminus y_1 C_n$, precisely:
\begin{equation}\label{partizione}
C_n\supset\bigsqcup^{n-1}_{k=0} y^k [C_n \setminus y_1 C_n].\end{equation}
 At first$$y_1^k [C_n \setminus y_1 C_n]\subset y_1^k C_n(y_2,y_3,\ldots,y_{s})\subset C_n,\;\;\; \forall k\in\{0,1,\ldots , n-1\};$$
in particular if $g\in y_1^k [C_n \setminus y_1 C_n]$ there exist $\widehat{i}_2,\ldots , \widehat{i}_{s}\in \{0,1,\ldots , n-1\} $ such that ${g=y_1^k y_2^{\widehat{i}_2}\ldots y_{s}^{\widehat{i}_{s}}}$. If $k\neq 0$ then $g\notin C_n\setminus y_1 C_n$, this implies:

 $$y_1^k [C_n\setminus y_1 C_n]\cap  [C_n\setminus y_1 C_n]=\emptyset, \;\;\;\forall k \in \{1,\ldots , n-1\}.$$ 

Thus $\{y_1^k [C_n\setminus y_1 C_n]\}_{k=0,\ldots , n-1}$ are disjoint sets and \eqref{partizione} is proved and therefore we deduce $\frac{|C_n\setminus y_1 C_n|}{|C_n|}\leq n^{-1}$.

\end{proof}

For a finite subset $Y\subset \Gamma$
 we may have different finite enumerations
 of $Y$, for example we consider $W,W'\subset \mathbb F_{ X}$,
 $W=\{w_1,\ldots , w_t\}$ and 
$W'=\{w'_1,\ldots , w'_{t'}\}$
 such that $\pi_{\Gamma,X}(W)=\pi_{\Gamma,X}(W')=Y$. In general,
$C_n(\pi_{\Gamma,X}(w_1),\ldots , \pi_{\Gamma,X}(w_t))\neq C_n(\pi_{\Gamma,X}(w'_1),\ldots , \pi_{\Gamma,X}(w'_{t'})) $ in $\Gamma$
 but these subsets are  both $n^{-1}$-invariant by left multiplication by every element
 $y\in Y$, by virtue of Lemma \ref{cubo}. By abuse of notation we simply write $C_n(Y)$ instead of $C_n(\pi(w_1),\ldots , \pi(w_t))$ when the choice of the finite preimage $W$ of $Y$ is irrelevant.\\
 Moreover, when the generating subset $X\subset \Gamma$ is clear from the context, we shall simply write $F_\Gamma$ (resp. $\mathfrak F\o l_{\Gamma}$) instead of $F_{\Gamma,X}$ (resp. $\mathfrak F\o l_{\Gamma,X}$).

\section{Kharlampovich groups}\label{K}

\begin{theorem}\label{semiabe}
Let $\Gamma=\langle L_1\cup L_2 \rangle$ be a finitely generated group, $L_1$ and $L_2$ two finite disjoint subsets and respectively $H_1$ and $H_2$ the subgroups that they generate.
Suppose that $H_2$ is amenable, $H_1^{\Gamma}$ is Abelian and $\Gamma= H_1^{\Gamma}\rtimes H_2$, then:
$$ A C_n(L_1^A)\in \mathfrak F\o l_{\Gamma}(n),\;\;\;\;\forall A\in \mathfrak F\o l_{H_2}(n).$$
where $L_1^A=\{a^{-1}xa:\;a\in A, x\in L_1\}.$
\end{theorem}
\begin{proof}
Set $B:=C_n(L_1^A)$, and observe that $|AB|=|A||B|$ since $A\subset H_2$ and $B\subset H_1^\Gamma$ and $H_2\cap H_1^{\Gamma}=\{1_\Gamma\}$.
\\For $x\in L_2$ we have:
$$\frac{|AB\setminus xAB|}{|AB|}\leq\frac{|A\setminus xA||B|}{|A||B|}\leq n^{-1},$$
For $x\in L_1$,  using Lemma \ref{cubo}, we have:

\vspace{3mm}

\noindent
$
 \dfrac{|AB\setminus xAB|}{|AB|}=\dfrac{|\{ab:\; a\in A,\; b\in B : ab\notin xAB\}|}{|A||B|} =$
 
 \vspace{3mm}
 \noindent
 $ = \dfrac{|\{ab:\; a\in A,\; b\in B : b\notin a^{-1}xAB\}|}{|A||B|}\leq  \dfrac{|\{ab:\;\; a\in A,\; b\in B : b\notin a^{-1}xaB\}|}{|A||B|}\leq $
 
 \vspace{3mm}
 \noindent
 $ \leq  \dfrac{|\bigcup_{a\in A} a(B\setminus a^{-1}xa B)|}{|A||B|}\leq n^{-1} $ (since $a^{-1}xa\in L_1^A$ and $B=C_n(L_1^A)$).
 \end{proof}

Consider the description of a Kharlampovich group $G(M)$ given in \cite{KMS}, with $M$ a Minsky machine with unsolvable halting problem and $p$ a fixed prime, using the same notation of \cite{KMS}, we have:
\begin{itemize}
\item $H_j:=\langle L_j \rangle,\; j=0,1,2, \;	\mbox{ and } H:=\langle L_1\cup L_2 \rangle,$
\item $H_j$ is abelian, \; $j=0,1,2$; 
\item $H_0,H_1$ are of exponent $p$;
\item $H_1^H$ is abelian of exponent $p$;
\item $H=H_1^{H}\rtimes H_2$;
\item $H_0^{G(M)}$ is abelian of exponent $p$;
\item $G(M)={H_0}^{G(M)}\rtimes H.$
\end{itemize}

Then by Theorem \ref{semiabe} we have:

$$C_n(L_2)\in \mathfrak F\o l_{H_2}(n),$$
$$C_n(L_2)C_n(L_1^{C_n(L_2)})\in \mathfrak F\o l_{H}(n),$$
since $H_1^H$ is of exponent $p$, then for $n\geq p$ we have $C_n=C_p$ in $H_1^{H}$ and the same holds in $H_0^{G(M)}$, thus:
\vspace{3mm}
$$C_n(L_2)C_p(L_1^{C_n(L_2)})C_p(L_0^{C_n(L_2)C_p(L_1^{C_n(L_2)})})\in \mathfrak F\o l_{G(M)}(n).$$
\vspace{3mm}

 The groups $G(M)$  have computable F\o lner sets: we have an algorithm with input $n$ and output a finite subset of the free group projecting onto an $n$-F\o lner set in $G(M)$. The Theorem \ref{KA} of Introduction is proved. Moreover, we have a bound from above for the cardinality of the smallest F\o lner sets for $G(M)$.
\begin{corollary}
The class of finitely presented groups with computable F\o lner sets is larger than the class of finitely presented amenable groups with solvable Word Problem.
\end{corollary}

\begin{corollary}
$$F_{G(M)}(n)\leq n^{|L_2|}p^{|L_1| n^{|L_2|}}p^{|L_0| n^{|L_2|}p^{|L_1| n^{|L_2|}}}.$$
\end{corollary}

\section{Abelian extension}\label{A}

We consider now the general Abelian extensions: a priori the procedure doesn't ensure computability of the F\o lner sets in every case.

\begin{theorem}\label{abex}
Let $\Gamma$ be finitely generated by $X$. Suppose $N \triangleleft \Gamma$ is an Abelian normal subgroup and denote by $\rho: \Gamma \rightarrow \Gamma / N$ the canonical projection. Then
$$A C_{2n |A|^2}(A^{-1}XA\cap N)\in  \mathfrak F\o l_{\Gamma, X}(n),$$
for each finite $ A\subset \Gamma$ such that $|A|=|\rho(A)|$ and $\rho(A) \in \mathfrak F\o l_{\Gamma/N, \rho(X)}(2n).$
\end{theorem}

\begin{proof}
Consider the finite set $S:=A^{-1}XA\cap N$ and, for each $x\in X$, the finite set $S_x:=A^{-1}xA\cap N$. We clearly have $|S|\leq |A|^2|X|$ and $|S_x|\leq |A|^2$.\\
Set $B:=C_{2n|A|^2}(S)\subset N$. Then by Lemma \ref{cubo} we have $\frac{|B\setminus s B|}{|B|}\leq (2n|A|^2)^{-1} \mbox{ for all } s\in S$;\\
thus for any $s\in S_x$, for any $x\in X$
\begin{equation}\label{1stella}
\frac{|B\setminus s B|}{|B|}\leq (2n|S_x|)^{-1}.
\end{equation}

Consider the set $F:=AB\subset \Gamma$ and notice that $|F|=|A||B|$ because the intersection $A\cap B$ has at most one element since $\rho_{|_A}$ is injective and $\rho$ sends $B$ to the identity of $\Gamma / N$.
So for $g\in F$ we write $g=ab$, $a\in A$, $b\in B$ in a unique way (again because $\rho_{|_A}$ is injective and $\rho(g)=\rho(a)$) and we write $A':=\rho(A)\subset \Gamma/N$, recall that this is $2n$-F\o lner in $\Gamma/N$.
\\For each $x\in X$, the set $F\setminus xF$ is the disjoint union of the subsets:
$$E^x_1=\{g\in F\setminus xF: \rho(g)\notin \rho(x) A'\}$$
$$E^x_2=\{g\in F\setminus xF: \rho(g)\in \rho(x) A'\}.$$
If $g=ab\in E^x_1$, since $\rho(g)=\rho(a)\notin \rho(x) A'$ we have $\rho(a)\in A'\setminus \rho(x) A'$. But $\rho$ is injective on $A$ then:
\begin{equation}\label{2stelle}
\frac{|E^x_1|}{|F|}=\frac{ |A'\setminus \rho(x) A'||B|}{|A||B|}\leq (2n)^{-1}.
\end{equation}

If $g=ab\in E^x_2$ then $\rho(a)\in\rho(x)A'=\rho(xA)$. Hence there exist $a'\in A, s\in N$ such that $as=xa'$. It follow that $s=a^{-1}xa'$ and $s\in S_x$. Now $g=xa's^{-1}b$,
and since $g\notin xF=xAB$ we necessarily have $b\notin sB$. Thus we have
$$\frac{|E^x_2|}{|F|}\leq\frac{|\{xa's^{-1}b,\,\, a'\in A,\,\, s\in S_x,\,\,b\in B\setminus sB\}|}{|A||B|}\leq \! \sum_{s\in S_x}\!\!\frac{|B\setminus sB|}{|B|}$$
And by \eqref{1stella}:
\begin{equation}\label{3stelle}
\frac{|E^x_2|}{|F|}\leq (2n)^{-1}.
\end{equation}

Combining \eqref{2stelle} and \eqref{3stelle} we deduce that $\frac{|F\setminus xF|}{|F|}=\frac{|E^x_1|}{|F|}+\frac{|E^x_2|}{|F|}\leq n^{-1}$, for any $x \in X$.
\end{proof}

\begin{corollary}
A finitely presented group which is the extension of an amenable group with solvable Word Problem by an Abelian group has computable F\o lner sets.
\end{corollary}
\begin{proof}
Consider the case of $\Gamma / N$ amenable with solvable Word Problem and with the set $\rho(X)$ as generators. If $\pi_{\Gamma / N}\colon \mathbb F_{X}\rightarrow \Gamma / N $ is the canonical epimorphism, for every $n$ we can compute $\mathcal A \in \mathbb F_X$ such that
 $\pi_{\Gamma / N}(\mathcal A)\in \mathfrak F\o l_{\Gamma/N, \rho(X)}(2n)$, but also with $|\mathcal A|=|\pi_{\Gamma / N}(\mathcal A)|$, by the solvability of the Word Problem.\\
 But then $A:=\pi_{\Gamma,X}(\mathcal A)$ is such that $\rho(A)=\pi_{\Gamma / N}(\mathcal A)\in \mathfrak F\o l_{\Gamma/N, \rho(X)}(2n)$ and $|A|=|\rho(A)|$, because:
$$|\rho(A)|\leq|A|\leq|\mathcal A|=|\pi_{\Gamma / N}(\mathcal A)|.$$ 
Moreover, given an element  $\omega\in \mathcal A^{-1}X\mathcal A$ we can compute if $\pi_{\Gamma / N}(\omega)=1_{\Gamma / N}$ or not, and then we can compute the preimage of $A^{-1}XA\cap N$ in $\mathbb F_X$ and finally we can compute a preimage of the $n$-F\o lner sets for $\Gamma$.
\end{proof}

This implies again that Kharlampovich groups have computable F\o lner sets, because they are Abelian extensions of  finitely presented metabelian, and therefore residually finite with solvable WP, groups.\\Notice that the Abelian group $N$ may be not finitely generated.

\begin{corollary}
If $\Gamma$ is finitely generated by $X$ and $N \triangleleft \Gamma$ is an Abelian normal subgroup, denoting with $\rho: \Gamma \rightarrow \Gamma / N$ the projection:
$$F_{\Gamma}(n)\leq F_{\Gamma/N}(2n) (2 n F_{\Gamma/N}(2n)^2)^{|X|F_{\Gamma/N}(2n)^2}.$$
\end{corollary}
\begin{proof}

We consider $\rho(A)\in \mathfrak F\o l_{\Gamma/N}(2n)$ such that $|\rho(A)|=|A|= F_{\Gamma/N}(2n)$, recall that $S=A^{-1}XA\cap N$ and then $|S|\leq |X||A|^2$. 
\end{proof}

\section{Splitting extensions}\label{S}
The situation is clearer if the extension splits. In this case we can also consider extensions by  amenable groups.

\begin{theorem}\label{semid}
Let $N$ and $H$ be groups respectively generated by the finite sets $Z$ and $Y$, let $\phi\colon H \rightarrow Aut(N)$ be a homomorphism.
Let $c:=\max\{|\phi_y(z)|_Z:\; z\in Z,\; y\in Y\}$.
\\Then if $A\in \mathfrak F\o l_{H,Y}(n)$ and $B\in \mathfrak F\o l_{N,Z}(nc^{|A|_Y})$  we have
$$AB\in \mathfrak F\o l_{N\rtimes_\phi H,Z\cup Y}(n),$$
(recall that $|A|_Y=\max\{|a|_Y:\;a\in A\}$).
\end{theorem}
\begin{proof}
We first observe that $|AB|=|A||B|$ because $A\subset H$ and $B\subset N$.
\\For $y\in Y$($\subset H$) we have:
$$\frac{|AB\setminus yAB|}{|AB|}\leq\frac{|A\setminus yA||B|}{|A||B|}\leq n^{-1}.$$
For $z\in Z$ ($\subset N$) we have:\\
$zab= aa^{-1}zab= a\phi_a(z)b$, so that $\{ab\in AB: zab\notin AB\} \subset \{ab\in AB: \phi_a(z)b\notin B\}$.
We deduce
$$\frac{|AB\setminus z AB|}{|AB|}\leq\frac{|\bigcup_{a\in A} a[B\setminus \phi_a(z) B]|}{|A||B|}\leq\frac{\sum_{a\in A}|B\setminus \phi_a(z) B|}{|A||B|}.$$
Since $|\phi_a(z)|_Z\leq c^{|a|_Y}\leq c^{|A|_Y}$ then, using Lemma \ref{catena}:
$$\frac{\sum_{a\in A}|B\setminus \phi_a(z)B|}{|A||B|}\leq \frac{|\phi_a(z)|_Z }{c^{|A|_Y} n}\leq n^{-1},$$
because $B\in \mathfrak F\o l_{N}(nc^{|A|_Y}).$

\end{proof}


\begin{corollary}
The semidirect product of two finitely generated groups with computable F\o lner sets has computable F\o lner sets.
\end{corollary}
\begin{proof}
We can compute $\mathcal A$, the preimage of a $n$-F\o lner set $A$ for $H$, we compute $m$, the maximal length of words in $\mathcal A$ in the free group. We compute $\mathcal B$, the preimage of $B\in \mathfrak F\o l_{N}(nc^{m})$. Since $|A|_Y\leq m$ we have $B\in \mathfrak F\o l_{N}(nc^{|A|_Y})$ and then by Theorem \ref{semid} we have that $\mathcal A\mathcal B$ is a preimages of an $n$-F\o lner set for the semidirect product.
\end{proof}

\begin{corollary}
In the same hypotheses of the above theorem:
$$F_{N\rtimes_\phi H}(n)\leq F_{H}(n|Y|) F_{N}(nc^{F_{H}(n|Y|)}).$$

\end{corollary}
\begin{proof}
By Lemma \ref{conn} we have $A\in \mathfrak F\o l_{H}(n)$ with $|A|_Y\leq |A|\leq F_H(|Y|n)$ then we choose the optimal $B\in \mathfrak F\o l_{N}(nc^{F_{H}(n|Y|)})$. Clearly $B\in \mathfrak F\o l_{N}(nc^{|A|_Y})$.
\end{proof}

\section{General extensions }\label{G}

\begin{theorem}\label{genex}${}$
Let $\Gamma$ be generated by the finite set $X$ and $N$ be a normal subgroup of $\Gamma$ generated by the finite set $Y$. Let $\rho: \Gamma\rightarrow K:=\Gamma /N$ be the projection to the quotient. For any finite subset $A\subset \Gamma$ such that 
$A':=\rho(A)\in \mathfrak F\o l_{K, \rho(X)}(2n)$, with $ |A|=|A'|$ and $|A|_X\leq|A'|_{\rho (X)}$,\\
and any $B\in \mathfrak F\o l_{N, Y}(2n|A'|^2 \Delta_N^{\Gamma}(2|A'|_{\rho(X)}+1))$ we have
$$AB\in \mathfrak F\o l_{\Gamma, X}(n).$$
\end{theorem}
\begin{proof}

Setting $F:=AB$ it is easy to see that $|F|=|A'||B|$ because $\rho$ is injective on $A$. 

For each $x\in X$, the set $F\setminus xF$ is the disjoint union of the sets $E^x_1$ and $E^x_2$, defined by:
$$E^x_1=\{g\in F\setminus xF: \rho(g)\notin \rho(x) A'\}$$
$$E^x_2=\{g\in F\setminus xF: \rho(g)\in \rho(x) A'\}.$$
We can write $g=ab$, with $a\in A$ and $b\in B$, in a unique way.

If $g\in E^x_1$, since $\rho(g)=\rho(a)\notin \rho(x) A'$ we have $\rho(a)\in A'\setminus \rho(x) A'$. Moreover, since $\rho$ is injective on $A$:
$$\frac{|E^x_1|}{|F|}=\frac{ |A'\setminus \rho(x) A'||B|}{|A'||B|}\leq (2n)^{-1}.$$

If $g\in E^x_2$ then $\rho(g)=\rho(a)\in\rho(x)A'$ so that there exists $a'\in A$ satisfying $\rho(a)=\rho(x) \rho(a')$. The images by $\rho$ of $a$ and $xa'$ are the same so we can find $s\in N$ such that $as=xa'$.

Setting $S_x:=A^{-1}xA\cap N$ we see that $s\in S_x$ and $|S_x|\leq |A|^2$. 
Then $g=xa's^{-1}b$, and since $g\notin xAB$ we deduce that $b\notin sB$. It follows that:

$$\frac{|E^x_2|}{|F|}\leq\frac{|\{xa's^{-1}b,\,\, a'\in A,\,\, s\in S_x,\,\,b\in B\setminus sB\}|}{|A'||B|}\leq \sum_{s\in S_x}\frac{|B\setminus sB|}{|B|}.$$
We have a bound for $|S_x|$; we need a bound for the length of the elements in $S_x$. For every $s\in S_x$ we have:\\
$|s|_Y\leq \Delta_N^{\Gamma}(|s|_X)$. On the other hand,  $|s|_X= |a^{-1}xa'|_X\leq 2|A|_X+1\leq 2|A'|_{\rho(X)}+1$.
\\From Lemma \ref{catena} we then deduce:
$$\frac{|B\setminus sB|}{|B|}\leq (2n|A'|^2)^{-1}\leq \frac{1}{2n|S_x|}.$$
Finally $\dfrac{|F\setminus x F|}{|F|}=\dfrac{|E^x_1|}{|F|}+\dfrac{|E^x_2|}{|F|}\leq n^{-1},$ showing that $F$ is the an $n$-F\o lner set.
\end{proof}

\begin{corollary}
Let $N,\Gamma,K$ finitely generated groups such that:
$$1\rightarrow N\rightarrow \Gamma \rightarrow K\rightarrow 1.$$
If $N$ has computable F\o lner sets, $\Delta_N^{\Gamma}$ is subrecursive, $K$ is amenable with solvable Word Problem, then $\Gamma$ has computable F\o lner sets.
\end{corollary}
\begin{proof}
$N$ and $K$ have computable F\o lner sets. For each $k$ we can construct $\mathcal A\subset \mathbb F_X$ such that $\pi_K(\mathcal A)\in  \mathfrak F\o l_{K}(k)$. We denote $A':=\pi_K(\mathcal A)$. If we consider $A:=\pi_{\Gamma,X}(\mathcal A)$, it is clear that $\rho(A)=A'\in \mathfrak F\o l_{K}(k)$. If $K$ has solvable Word Problem we can detect $\mathcal A$ such that $\pi_K$ is injective on $\mathcal A$ and $|\omega|=|\pi_K(\omega)|_{\rho(X)}$ for every $\omega \in \mathcal A$. So we can compute a preimage for a set $A$ respecting the hypotheses of the Theorem \ref{genex}. For the set $B$ we just need the computability (of a bound) of the number $2n|A'|^2\Delta_N^G(2|A'|_{\rho(X)}+1)$, so if $\Delta_N^{\Gamma}$ is subrecursive we have the thesis.
\end{proof}

Finally, from Theorem \ref{genex} and again using Lemma \ref{conn}:
\begin{corollary}
Let $N,\Gamma,K$ finitely generated groups such that:
$$1\rightarrow N\rightarrow \Gamma \rightarrow K\rightarrow 1.$$
Then
$$ F_{\Gamma}(n)\leq F_{K}(|X|n) F_N(2n F_{K}(|X|n)^2 \Delta_N^{\Gamma}(2F_{K}(|X|n)+1)).$$
Thus if $N$ and $K$ have subrecursive F\o lner function  and if $\Delta_N^\Gamma$ is subrecursive then $\Gamma$ has subrecursive F\o lner function as well.

\end{corollary}

\bibliographystyle{amsxport}

\end{document}